\documentclass[12pt, a4paper]{article}
\usepackage{amsthm,amsfonts,amsmath,amssymb,bookmark,enumerate,appendix}

\newtheorem{definition}{Definition}
\newtheorem{theorem}{Theorem}
\newtheorem{remark}{Remark}

\newtheorem{lemma}{Lemma}

\def\geqslant{\ge}
\def\leqslant{\le}
\def\bq{\begin{eqnarray}}
\def\eq{\end{eqnarray}}
\def\bqq{\begin{eqnarray*}}
\def\eqq{\end{eqnarray*}}
\def\nn{\nonumber}

\def\eps{\varepsilon}

\newcommand{\h}{\hspace*{.24in}}

\def\R {\mathbb{R}}
\def\C {\mathbb{C}}
\def\N {\mathbb{N}}
\def\F {\mathcal{F}}
\def\eps {\varepsilon}
\def\div{\operatorname{div}}

\title{\bf The body force in a three-dimensional Lam\'e system: identification and regularization}
\author{Dang Duc Trong, Phan Thanh Nam and Phung Trong Thuc}

\begin{document}
\date{September 16, 2011}
\maketitle

\begin{abstract} Let a three-dimensional isotropic elastic body be described by the Lam\'e system with the body force of the form $F(x,t)=\varphi(t)f(x)$, where $\varphi$ is known. We consider the problem of determining the unknown spatial term $f(x)$ of the body force when the surface stress history is given as the overdetermination. This inverse  problem is ill-posed. Using the interpolation method and truncated Fourier series, we construct a regularized solution from approximate data and provide explicit error estimates. 
\vspace{5pt}


\end{abstract} 


\section{Introduction}
Let $\Omega=(0,1)\times(0,1)\times (0,1)$ represent a three-dimensional isotropic elastic body and let $T>0$ be the length of the observation time. For each $x:=(x_1,x_2,x_3)\in \Omega$, we denote by $u(x,t)=(u_1(x,t),u_2(x,t),u_3(x,t))$ the displacement, where $u_j$ is the displacement in the $x_j$-direction. 

As known, $u$ satisfies the Lam\'e system (see, e.g., \cite{MHS,TG}) 
\bq
\frac{{\partial^2 u}}
{{\partial t^2 }}- \mu \Delta u -\left( {\lambda  +\mu } \right)\nabla \left( {\div (u)} \right) = F,~~x\in \Omega, t\in (0,T),\label{1}
\eq
where $F(x,t)=(F_1(x,t),F_2(x,t),F_3(x,t))$ is the body force and $\div(u)=\nabla \cdot u =\partial u_1/\partial x_1+\partial u_2/\partial x_2+\partial u_3/\partial x_3$. The Lam\'e constants $\lambda$ and $\mu$ satisfy $\mu>0$ and $\lambda+2\mu>0$. The system (\ref{1}) is associated with the initial condition 
\bq
\left\{ \begin{gathered}
  (u_1(x,0),u_2(x,0),u_3(x,0))=(g_{1}(x),g_{2}(x),g_{3}(x)),~x\in \Omega,\hfill \\
   \left( {\frac{\partial u_1}{\partial t}(x,0),\frac{\partial u_2}{\partial t}(x,0),\frac{\partial u_3}{\partial t}(x,0)}\right)
=(h_{1}(x),h_{2}(x),h_{3}(x)),~x\in \Omega,\hfill \\
 \end{gathered}  \right.\label{3}
 \eq
 and the Dirichlet boundary condition 
\bq (u_1(x,t),u_2(x,t),u_3(x,t))=(0,0,0),\h x\in \partial \Omega, t\in (0,T),\label{2}\eq
namely the boundary of the elastic body is clamped. 

The direct problem is to determine $u$ from $u(0,x), u_t(0,x)$ and $F$. We are, however, interested in the inverse problem of determining both of $(u,F)$. Of course, to ensure the uniqueness of the solution we shall require some additional information (the {\it overdetermination}). Similarly to \cite{MMM,DPPT}, we shall assume that the surface stress is given on the boundary of the body, i.e.,
\bq
\left( \begin{gathered}
  \sigma _1 \h \tau _{{\text{12}}} \h\tau _{{\text{13}}}  \hfill \\
  \tau _{{\text{21}}}\h\sigma _{\text{2}} \h\tau _{{\text{23}}}  \hfill \\
  \tau _{{\text{31}}} \h\tau _{{\text{32}}} \h\sigma _{\text{3}}  \hfill \\ 
\end{gathered} \right)\left(\begin{gathered}
  \text{\bf n}_1  \hfill \\
 \text{\bf n}_2  \hfill \\
 \text{\bf n}_3  \hfill \\ 
\end{gathered}  \right) = \left( \begin{gathered}
  X_1  \hfill \\
  X_2  \hfill \\
  X_3  \hfill \\ 
\end{gathered}  \right)~,~ x\in \partial\Omega,t\in (0,T),\label{4}
\eq
where  $ \text{\bf n}=( \text{\bf n}_1, \text{\bf n}_2, \text{\bf n}_3)$ is the outward unit normal vector of $\partial \Omega$ and the stresses $\sigma$ and $\tau$ are defined by
$$
  \tau_{jk}  = \mu \left( {\frac{{\partial u_j }}
{{\partial x_k }} + \frac{{\partial u_k }}
{{\partial x_j }}} \right), ~ \sigma _j  = \lambda \div(u)+ 2\mu \frac{{\partial u_j }}
{{\partial x_j }},\h j,k\in\{1,2,3\}.
$$

In 2005, Grasselli, Ikehata and Yamamoto \cite{MMM} showed that the body force of the form
 $F(x,t)=\varphi(t)f(x)$ is uniquely determined from (\ref{1}-\ref{4}) provided that
 $\varphi\in C^1([0,T])$ is given such that $\varphi(0)\ne 0$ and the time of observation
 $T>0$ is large enough. In spite of the uniqueness, the problem of determining the spatial term $f$ is still ill-posed,
 i.e. a small error of data may cause a large error of solutions.
 Therefore, it is important in practice to find a regularization process, namely to construct an approximate solution using approximate data. 

Recently, the regularization problem was solved partially in \cite{DPPT}, where a regularized solution for the time-independent term
 $f$ is produced using further information on the final condition $u(x,T)$.
 The final condition plays an essential role in \cite{DPPT} since it enables the authors to find an explicit formula for the Fourier transform of $f$, and then use this information to recover $f$.
 
It was left as an open problem in \cite{DPPT} (see their Conclusion) to
 find a regularization process without using this technical condition. The aim of the present paper is to solve this problem completely, i.e. to find a regularization process of $f$ using only the data in (\ref{1}-\ref{4}). We follow the interpolation method introduced in \cite{TDN} where the authors constructed a regularized solution for the heat source of a heat equation. More precisely, lacking the final condition, we are only able to find an approximation for the Fourier transform $\widehat f(\xi)$ with $|\xi|$ large. The idea is that because  $\widehat f(\xi)$ is an analytic function (since $f$ has compact support), we can use some interpolation process to recover $\widehat f(\xi)$ with $|\xi|$ small. 

The paper is organized as follows. In Section 2, we shall set some  notations and state our main results. Then we prove the uniqueness in Section 3 and the regularization in Section 4. Finally, in Section 5 we test our regularization process on an explicit numerical example.  

\section{Main results}

Recall that our aim is to recover the spatial term
$$
f(x)=(f_1(x),f_2(x),f_3(x)), x\in \Omega,
$$
of the body force $F(x,t)=\varphi(t)f(x)$ from the system (\ref{1}-\ref{4}). The Lam\'e constants always satisfy $\mu>0$ and $\lambda+2\mu>0$, and the data $I = (\varphi ,X,g,h)$ is allowed to be non-smooth,  
\[
I \in \left( {L^1(0,T),(L^1(0,T,L^1(\partial \Omega)))^3, (L^2(\Omega))^3,(L^2(\Omega))^3} \right).
\]

For $\xi  = (\xi _1 ,\xi _2 ,\xi _3 ),\zeta  = (\zeta _1 ,\zeta _2 ,\zeta _3 )\in \C^3$, we set $\xi \cdot\zeta=\xi _1 \zeta _1  + \xi _2 \zeta _2  + \xi _3 \zeta _3 $, $\left| \xi  \right| = \sqrt {\overline{\xi}\cdot\xi }$ and $\left| \xi  \right|_0 = \sqrt {|\xi \cdot\xi| }$. For $\alpha \in\C^3$ and $x\in \R^3$, denote 
\bqq
  G(\alpha, x)&=&G_1^1\left( \alpha ,x \right) =G_2^2\left( \alpha ,x \right)=G_3^3\left( \alpha , x \right)\hfill\\
&=&\cos \left( {{\alpha _1}{x_1}} \right)\cos \left( {{\alpha _2}{x_2}} \right)\cos \left( {{\alpha _3}{x_3}} \right), \hfill \\
  G_2^1\left( \alpha ,x \right) &=& G_1^2\left( \alpha , x \right)= - \sin \left( {{\alpha _1}{x_1}} \right)\sin \left( {{\alpha _2}{x_2}} \right)\cos \left( {{\alpha _3}{x_3}} \right), \hfill \\
  G_3^1\left( \alpha ,x \right) &=& G_1^3\left( \alpha , x \right)=  - \sin \left( {{\alpha _1}{x_1}} \right)\cos \left( {{\alpha _2}{x_2}} \right)\sin \left( {{\alpha _3}{x_3}} \right), \hfill \\
    G_3^2\left( \alpha ,x \right) &=& G_2^3\left( \alpha ,x \right)= - \cos \left( {{\alpha _1}{x_1}} \right)\sin \left( {{\alpha _2}{x_2}} \right)\sin \left( {{\alpha _3}{x_3}} \right).
\eqq
Sometimes we shall write ${G_k^j}$ instead of ${G_k^j}( \alpha,x)$ if there is no confusion.

We start with the following lemma.

\begin{lemma}\label{bod1} If $u\in (C^2([0,T];L^2(\Omega))\cap L^2(0,T;H^2(\Omega)))^3$ and $f\in (L^2(\Omega))^3$ satisfy the system (\ref{1}-\ref{4}) with data $I = (\varphi ,X,g,h)$, then 
$$
\int\limits_\Omega  {{f_j}} {G}dx = \frac{{{E_{1j}}(I)\left( \alpha  \right) + {E^*_{1j}}(\alpha )}}{{{D_1}(I)\left( \alpha  \right)}} + \frac{{{E_{2j}}(I)\left( \alpha  \right) + {E^*_{2j}}(\alpha )}}{{{D_2}(I)\left( \alpha  \right)}},~j \in \{1,2,3\}
$$
for all $\alpha=(\alpha_1,\alpha_2,\alpha_3)\in \C^3$ such that
$$ |\alpha|_0^2=-(\alpha_1^2+\alpha_2^2+\alpha_3^2) \ge 0~~\text{and}~{D_1}\left( I \right)\left( \alpha  \right)\ne 0,{D_2}\left( I \right)\left( \alpha  \right) \ne 0,$$
where
\bqq 
  {D_1}\left( I \right)\left( \alpha  \right) &=& -|\alpha|_0^2\int\limits_0^T {\varphi \left( t \right)\frac{{\sinh \left( {\sqrt {  \left( {\lambda  + 2\mu } \right)} \,{{\left| \alpha  \right|}_0}\left( {T - t} \right)} \right)}}
{{\cosh \left( {\sqrt {  \left( {\lambda  + 2\mu } \right)}\, {{\left| \alpha  \right|}_0}T} \right)}}} dt, \hfill \\[0.5cm]
{D_2}\left( I \right)\left( \alpha  \right) &=& -|\alpha|_0^2 \int\limits_0^T {\varphi \left( t \right)\frac{{\sinh \Bigl( {\sqrt {  \mu }\,{{\left| \alpha  \right|}_0}\left( {T - t} \right)} \Bigr)}}
{{\cosh \Bigl( {\sqrt {  \mu }\,{{\left| \alpha  \right|}_0}T} \Bigr)}}} dt, \hfill \\
{E_{1j}}(I)(\alpha ) &=&  - \sqrt {\lambda  + 2\mu }\,{\left| \alpha  \right|_0}{\alpha _j}\int\limits_\Omega {\left( {{\alpha _1}{g_1}{G_1^j} + {\alpha _2}{g_2}{G_2^j} + {\alpha _3}{g_3}{G_3^j}} \right)} dx  \hfill \\ 
&~&-\tanh \left({\sqrt {\lambda + 2\mu } \,{\left| \alpha  \right|_0}T} \right){\alpha _j}\int\limits_\Omega  {\left( {{\alpha _1}{h_1}{G_1^j} + {\alpha _2}{h_2}{G_2^j} + {\alpha _3}{h_3}{G_3^j}} \right)} dx  \hfill \\ 
&~&- \int\limits_0^T {\int\limits_{\partial \Omega } {\dfrac{{\sinh \Bigl( {\sqrt {\lambda + 2\mu }\,{\left| \alpha  \right|_0}\left( {T - t} \right)} \Bigr)}}{{\cosh \Bigl( {\sqrt {\lambda  + 2\mu }\,{ \left| \alpha  \right|_0}T} \Bigr)}}}} {\alpha _j}\left( {{\alpha _1}{X_1}{G_1^j} +{\alpha _2}{X_2}{G_2^j} + {\alpha _3}{X_3}{G_3^j}} \right) d\omega dt,
\eqq
\bqq
E_{1j}^{*}\left(\alpha\right) &=& \frac{\sqrt{\lambda +2\mu}\,{\left|\alpha\right|_0}}{\cosh\left(\sqrt{\lambda +2\mu}\,{\left|\alpha\right|_0}T\right)}\alpha_{j}\times \hfill \\ 
&~&\times\intop_{\Omega}\Bigl(\alpha_{1}u_{1}\left(x,T\right)G_{1}^j+\alpha_{2}u_{2}\left(x,T\right)G_{2}^j+\alpha_{3}u_{3}\left(x,T\right)G_{3}^j\Bigr)\, dx,\hfill\\
E_{2j}\left(I\right)\left(\alpha\right) &=& -\sqrt{\mu}\,{\left|\alpha\right|_0}\intop_{\Omega}\Bigl( -\left| \alpha\right|_0^2g_{j}G_{j}^j - \alpha_{j}\bigl(\alpha_{1}g_{1}G_{1}^j+\alpha_{2}g_{2}G_{2}^j+\alpha_{3}g_{3}G_{3}^j\bigr)\Bigr) \,dx \hfill\\
&~&-\tanh\left(\sqrt{\mu}\,{\left|\alpha\right|_0}T\right)\intop_{\Omega}\Bigl( - \left| \alpha  \right|_0^2h_{j}G_{j}^j- \alpha_{j}\bigl(\alpha_{1}h_{1}G_{1}^j+\alpha_{2}h_{2}G_{2}^j+\alpha_{3}h_{3}G_{3}^j\bigr)\Bigr)\, dx \hfill\\
&~&+\intop_{0}^{T}\intop_{\partial\Omega}\frac{\sinh\bigl(\sqrt{\mu}\,{\left|\alpha\right|_0}\left(T-t\right)\bigr)}{\cosh\bigl(\sqrt{\mu}\,{\left|\alpha\right|_0}T\bigr)}\times\hfill \\
&~&~~~~\times\Bigl(  \left| \alpha  \right|_0^2X_{j}G_{j}^j+\alpha_{j}\bigl(\alpha_{1}X_{1}G_{1}^j+\alpha_{2}X_{2}G_{2}^j+\alpha_{3}X_{3}G_{3}^j\bigr)\Bigr)\, d\omega dt,\hfill\\
E_{2j}^{*}\left(\alpha\right) &=& \frac{\sqrt{\mu}\,{\left|\alpha\right|}_0}{\cosh\left(\sqrt{\mu}\,{\left|\alpha\right|_0}T\right)}\intop_{\Omega}\biggl( - \left| \alpha  \right|_0^2u_{j}\left(x,T\right)G_{j}^j + \hfill\\
&~&-\, \alpha_{j}\Bigl(\alpha_{1}u_{1}\left(x,T\right)G_{1}^j+\alpha_{2}u_{2}\left(x,T\right)G_{2}^j+\alpha_{3}u_{3}\left(x,T\right)G_{3}^{j}\Bigr)\biggl)\,dx.
\eqq
\end{lemma}

Note that $E^*_{1j}$ and $E^*_{2j}$ in  Lemma \ref{bod1} depend on $u(x,T)$ instead of the data $I = (\varphi ,X,g,h)$. Therefore, in general these terms are unknown. However, our observation is that with $|\alpha|$ large, $E^*_{1j}$ and $E^*_{2j}$ are relatively small in comparison with $E_{1j}$ and $E_{2j}$, and can be relaxed when computing the integrals $\int\limits_\Omega  {f_jGdx}$ which are the Fourier coefficients of $f_j$'s. So we introduce some convenient notations.

\begin{definition}[Information from data] For $I = (\varphi ,X,g,h)$ and $\alpha\in \C^3$ such that $\alpha\cdot \alpha<0$, denote (using notations of Lemma \ref{bod1}).
\[
H_j(I)(\alpha ) = \left\{ \begin{gathered}
  \frac{{E_{1j} (I)}}
{{D_1 (I)}} + \frac{{E_{2j} (I)}}
{{D_2 (I)}},\h\text{ if } D_1(I)(\alpha).D_2(I)(\alpha)\ne 0, \hfill \\
  0,\h\h\h\h\h~~\text{ if } D_1(I)(\alpha).D_2(I)(\alpha)= 0. \hfill \\ 
\end{gathered}  \right.
\]
\end{definition}

\begin{definition}[Fourier coefficients] For  $\alpha  = \left( {{\alpha _1},{\alpha _2},{\alpha _3}} \right) \in {\C^3}$, $w \in {L^2}\left( \Omega  \right)$, denote  
$${\F}(w)\left( \alpha  \right) =\int\limits_\Omega  {w(x)G(\alpha,x)dx} = \int\limits_\Omega  {w(x) \cos \left( {{\alpha _1}{x_1}} \right)\cos \left( {{\alpha _2}{x_2}} \right)\cos \left( {{\alpha _3}{x_3}} \right)dx}.
$$
\end{definition}

Note that any function $w\in L^2(\Omega)$ admits the representation 
\bq
w(x) = \sum\limits_{m,n,p \geqslant 0} {\kappa(m,n,p){\F}\left( w \right)(m\pi,n\pi,p\pi)\cos (m\pi x_1)\cos (n\pi x_2)\cos (p\pi x_3)}, \label{Fs}
\eq
where  $\kappa(m,n,p):=(1+1_{\{m\ne 0\}})(1+1_{\{n\ne 0\}})(1+1_{\{p\ne 0\}})$. 

As we explained above, we hope to approximate $\F(f_j)(\alpha)$ by $H_j(I)(\alpha)$ with $|\alpha|$ large. To do this, we need some lower bounds on $|D_1(I)(\alpha)|$ and $|D_2(I)(\alpha)|$. We shall require the following assumptions on $\varphi$ and $T$. 

$\bf{(W1)}$ There exist $\Lambda(\varphi)\in(0,T)$ and $C({\varphi})>0$ such that either $\varphi(t)\ge C({\varphi})$ for a.e $t\in (0,\Lambda(\varphi))$ or $\varphi(t)\le -C({\varphi})$ for a.e $t\in (0,\Lambda(\varphi))$.

$\bf{(W2)}$ $T>2\max\left\{ \dfrac{1}{\sqrt{\mu }},\dfrac{1}{\sqrt{\lambda +2\mu }} \right\}$, or $\bf{(W2')}$ $T>12\sqrt 5 \times \max\left\{ \dfrac{1}{\sqrt{\mu }},\dfrac{1}{\sqrt{\lambda +2\mu }} \right\}$.
\begin{remark} If $\varphi$ is continuous at $t=0$ then the condition (W1) is equivalent to $\varphi(0)\ne 0$. The conditions (W2) and (W2') mean that the observation time must be long enough. The condition (W2) is enough for the uniqueness, while the stronger condition (W2') is required in our regularization. These conditions should be compared to similar conditions (2.7) and (2.8) in $\cite{MMM}$.  
\end{remark}

\begin{theorem}[Uniqueness] \label{duynhatnghiem} Assume that (W1) and (W2) hold true. Then the system (\ref{1}-\ref{4}) has at most one solution 
$$(u,f)\in \left( {(C^2([0,T];L^2(\Omega))\cap L^2(0,T;H^2(\Omega)))^3,(L^2(\Omega))^3} \right).$$
\end{theorem}

The main point in our regularization is to recover $\F(f)(\alpha)$ with $|\alpha|$ small from approximate values of $\F(f)(\alpha)$ with $|\alpha|$ large. As in \cite{TDN} we shall use the Lagrange interpolation polynomial.

\begin{definition}[Lagrange interpolation polynomial] Let $A = \left\{ {{x_1},{x_2},...,{x_n}} \right\}$ be a set of $n$ mutually distinct complex numbers and let $w$ be a complex function. Then the Lagrange interpolation polynomial $L\left[ {A;w} \right]$ is
$$L\left[ {A;w} \right]\left( z \right) = \sum\limits_{j = 1}^n {\left( {\prod\limits_{k \ne j} {\frac{{z - {x_k}}}{{{x_j} - {x_k}}}} } \right)w\left( {{x_j}} \right)}.$$
\end{definition}

\begin{theorem} [Regularization] \label{ch2chinhhoanghiem}Assume that $\left( {{u^0},{f^0}} \right)$ is the exact solution to the system (\ref{1}-\ref{4}) with the exact data ${I_0}=\left( {{\varphi ^0},{X^0},{g^0},{h^0}} \right)$, where the conditions (W1) and (W2') are satisfied. Let $\varepsilon  > 0$, consider the inexact data ${I_\varepsilon }=\left( {{\varphi ^\varepsilon },{X^\varepsilon },{g^\varepsilon },{h^\varepsilon }} \right)$ such that, for all $j\in \{1,2,3\}$,
$$
{\left\| {{\varphi ^0} - {\varphi ^\varepsilon }} \right\|_{{L^1}}} \leqslant \varepsilon,{\left\| {X_j^\varepsilon  - X_j^0} \right\|_{{L^1}}} \leqslant \varepsilon ,{\left\| {g_j^\varepsilon  - g_j^0} \right\|_{{L^2}}} \leqslant \varepsilon,
  {\left\| {h_j^\varepsilon  - h_j^0} \right\|_{{L^2}}} \leqslant \varepsilon. 
$$
Construct the regularized solution $f_j^\varepsilon$ from ${I_\varepsilon }=\left( {{\varphi ^\varepsilon },{X^\varepsilon },{g^\varepsilon },{h^\varepsilon }} \right)$ by
$$f_j^\varepsilon = \sum\limits_{0\, \leqslant m,\,n,\,p \,\leqslant \,{r_\varepsilon }} {\kappa\left( {m,n,p} \right)} {\F_\varepsilon }\left( {m,n,p} \right){G}\left( {m\pi ,n\pi ,p\pi } \right), $$
where
\bqq
r_\varepsilon  &=& Z \cap \left( {\dfrac{{\ln \left( {{\varepsilon ^{ - 1}}} \right)}}{{60}},\dfrac{{\ln \left( {{\varepsilon ^{ - 1}}} \right)}}{{60}} + 1} \right],\hfill\\
B_{{r_\varepsilon }} &=& \left\{ { \pm \left( {5{r_\varepsilon } + j}\right):j = 1,2,...,24{r_\varepsilon }} \right\},\hfill\\
\F_\varepsilon \left( {m,n,p} \right) &=& L\left[ {{B_{{r_\varepsilon }}},H_j(I_\varepsilon)\left( {-i.,n\pi ,p\pi } \right)} \right]\left( {im\pi } \right).
\eqq
Then we have, for all $j \in \left\{ {1,2,3} \right\}$, 
\begin{itemize}
\item[(i)] (Convergence) $f_j^\varepsilon  \in {C^\infty }\left( {{\R^3}} \right)$ and $f_j^\varepsilon  \to f_j^0$ in $L^2(\Omega)$ as $\eps\to 0$.

\item[(ii)] ($L^2$-estimate) If $f_j^0 \in H^1\left( \Omega  \right)$ then $f_j^{\varepsilon}  \to f_j^0$ in $H^1(\Omega)$ and there exist constants ${\varepsilon _0} > 0$ and ${C_0} > 0$ depending only on the exact data such that for all $\varepsilon  \in \left( {0,{\varepsilon _0}} \right)$,
$${\left\| {f_j^\varepsilon  - f_j^0} \right\|_{{L^2}\left( \Omega  \right)}} \leqslant {C_0}{\Bigl( {\ln \left( {{\varepsilon ^{ - 1}}} \right)} \Bigr)^{ -\frac{ 1}{2}}}.$$

\item[(iii)] ($H^1$-estimate) If ${f_j^0} \in H^2\left( \Omega  \right)$ then for all $\varepsilon  \in \left( {0,{\varepsilon _0}} \right)$,
$${\left\| {f_j^\varepsilon  - f_j^0} \right\|_{{H^1}\left( \Omega  \right)}} \leqslant {C_0}{\Bigl( {\ln \left( {{\varepsilon ^{ - 1}}} \right)} \Bigr)^{ - \frac{1}{4}}}.$$
\end{itemize}

\end{theorem}

 \begin{remark} In our construction, the convergence in $H^2(\Omega)$ is not expected even if $f^0_j\in C^{\infty}({\overline{\Omega}})$ since $\partial {f_j^{\varepsilon}}/ \partial \text{\bf n}=0$ on $\partial\Omega$.
\end{remark}

\section{Uniqueness}

In this section we shall prove Theorem \ref{duynhatnghiem}. We start with the proof of Lemma \ref{bod1} by using the argument in \cite{DPPT}. 
\begin{proof} [Proof of Lemma \ref{bod1}] Fix $\alpha=(\alpha_1,\alpha_2,\alpha_3)\in \C^3$ with $|\alpha|_0^2=-( \alpha_1^2+\alpha_2^2+\alpha_3^2) \ge 0$. For any $j \in \left\{ {1,2,3} \right\}$, getting the inner product (in $L^2(\Omega)$) of the k-th equation of the system  (\ref{1}) with ${G_k^j}$ ($k = 1,2,3$), then using the integral by part and the boundary conditions \eqref{2} and \eqref{4}, we have
\begin{align*}
\phantom{=}& \frac{{{d^2}}}
{{d{t^2}}}\int\limits_\Omega  {{u_k}{G_k^j}dx + \mu {{\left( {\alpha _1^2 + \alpha _2^2 + \alpha _3^2} \right)}}\int\limits_\Omega  {{u_k}{G_k^j}dx} } \\
 \phantom{=}& + \left( {\lambda  + \mu } \right){\alpha _k}\int\limits_\Omega  {\left( {{\alpha _1}{u_1}{G_1^j} + {\alpha _2}{u_2}{G_2^j} + {\alpha _3}{u_3}{G_3^j}} \right)} dx \hfill \\
   =&  \int\limits_{\partial \Omega } {{X_k}{G_k^j}d\omega  - \varphi \left( t \right)\int\limits_\Omega  {{f_k}{G_k^j}dx.} } \tag{$A_k$} \label{eq:Ak}
\end{align*}
Multiplying \eqref{eq:Ak} with ${\alpha _j}{\alpha _k}$, and then getting the sum of $k=1,2,3$, we have
\bq\label{PT6}
&~& \frac{{{d^2}}}
{{d{t^2}}}\int\limits_\Omega  {\left( {{\alpha _j}{\alpha _1}{u_1}{G_1^j} + {\alpha _j}{\alpha _2}{u_2}{G_2^j} + {\alpha _j}{\alpha _3}{u_3}{G_3^j}} \right)dx} \nn\hfill \\
&~&- \left( {\lambda  + 2\mu } \right)|\alpha|_0^2\int\limits_\Omega  {\left( {{\alpha _j}{\alpha _1}{u_1}{G_1^j} + {\alpha _j}{\alpha _2}{u_2}{G_2^j} + {\alpha _j}{\alpha _3}{u_3}{G_3^j}} \right)} dx \nn\hfill \\
&=&  \int\limits_{\partial \Omega } {\left( {{\alpha _j}{\alpha _1}{X_1}{G_1^j} + {\alpha _j}{\alpha _2}{X_2}{G_2^j} + {\alpha _j}{\alpha _3}{X_3}{G_3^j}} \right)d\omega } \nn \hfill \\
&~& + \,\,\varphi \left( t \right)\int\limits_\Omega  {\left( {{\alpha _j}{\alpha _1}{f_1}{G_1^j} + {\alpha _j}{\alpha _2}{f_2}{G_2^j} + {\alpha _j}{\alpha _3}{f_3}{G_3^j}} \right)dx.}
\eq
Choosing $k=j$ in \eqref{eq:Ak}, then multiplying the result by $|\alpha|_0^2$, and adding to \eqref{PT6}, we obtain
\bq \label{PT7}
&~& \frac{{{d^2}}}
{{d{t^2}}}\int\limits_\Omega  {\Bigl( { {|\alpha|_0 ^2}\,{u_j}{G} + \bigl( {{\alpha _j}{\alpha _1}{u_1}{G_1^j} + {\alpha _j}{\alpha _2}{u_2}{G_2^j} + {\alpha _j}{\alpha _3}{u_3}{G_3^j}} \bigr)} \Bigr)} dx \nn \hfill\\
&~& + \mu |\alpha|_0 ^2 \int\limits_\Omega  {\Bigl( {{{ -|\alpha |_0 ^2}}\,{u_j}{G_j^j} - \bigl( {{\alpha _j}{\alpha _1}{u_1}{G_1^j} + {\alpha _j}{\alpha _2}{u_2}{G_2^j} + {\alpha _j}{\alpha _3}{u_3}{G_3^j}} \bigr)} \Bigr)} dx \nn\hfill\\
   &=&  \int\limits_{\partial \Omega } {\Bigl( { {{|\alpha|_0 ^2}}{X_j}{G_j^j} + \bigl( {{\alpha _j}{\alpha _1}{X_1}{G_1^j} + {\alpha _j}{\alpha _2}{X_2}{G_2^j} + {\alpha _j}{\alpha _3}{X_3}{G_3^j}} \bigr)} \Bigr)d\omega } \nn\hfill\\
&~&+ \varphi \left( t \right)\int\limits_\Omega  {\Bigl( {{{|\alpha|_0 ^2}}{f_j}{G_j^j} +\bigl( {{\alpha _j}{\alpha _1}{f_1}{G_1^j} + {\alpha _j}{\alpha _2}{f_2}{G_2^j} + {\alpha _j}{\alpha _3}{f_3}{G_3^j}} \bigr)} \Bigr)dx.}  
\eq

We can consider \eqref{PT6} and \eqref{PT7} as the differential equations of the form
\begin{equation}
y''\left( t \right) - {\eta ^2}y\left( t \right) = h\left( t \right), \label{viphan}
\end{equation}
where $\eta >0$ is independent of $t$. Getting the inner product (in $L^2(0,T)$) of \eqref{viphan} with $\frac{{\sinh \left( {\eta \left( {T - t} \right)} \right)}}{{\cosh \left( {\eta T} \right)}}$, we have 
\begin{equation}\label{kqviphan}
- y'\left( 0 \right)\tanh \left( {\eta T} \right) - \eta y\left( 0 \right) + \frac{\eta }
{{\cosh \left( {\eta T} \right)}}y\left( T \right) = \int\limits_0^T {h\left( t \right)\frac{{\sinh \left( {\eta \left( {T - t} \right)} \right)}}{{\cosh \left( {\eta T} \right)}}} dt.
\end{equation}

Applying \eqref{kqviphan} to \eqref{PT6} and \eqref{PT7}  with $\eta  = \sqrt {  \lambda + 2\mu }\,{\left| \alpha  \right|_0}$ and $\eta  = \sqrt { \mu }\,{\left| \alpha  \right|_0}$ respectively, we get 
\begin{eqnarray*}
E_{1j}\left(I\right)\left(\alpha\right)+E_{1j}^{*}\left(\alpha\right) & = & \frac{1}{{ - \left| \alpha  \right|_0^2}}D_{1}\left(I\right)\left(\alpha\right)\intop_{\Omega}\alpha_{j}\left(\alpha_{1}f_{1}G_{1}^j+\alpha_{2}f_{2}G_{2}^j+\alpha_{3}f_{3}G_{3}^j\right)dx,\\
E_{2j}\left(I\right)\left(\alpha\right)+E_{2j}^{*}\left(\alpha\right) & = & \frac{1}{{ - \left| \alpha  \right|_0^2}}D_{2}\left(I\right)\left(\alpha\right)\times\\
 &  & \times\intop_{\Omega}\Bigl({ - \left| \alpha  \right|_0^2}f_{j}G_{j}^j-\alpha_{j}\bigl(\alpha_{1}f_{1}G_{1}^j+\alpha_{2}f_{2}G_{2}^j+\alpha_{3}f_{3}G_{3}^j\bigr)\Bigl)dx.\end{eqnarray*}
It follows from the latter equations that
$$\int\limits_\Omega  {{f_j}} {G}dx = \frac{{{E_{1j}}(I)\left( \alpha  \right) + {E_{1j}}^*(\alpha )}}{{{D_1}(I)\left( \alpha  \right)}} + \frac{{{E_{2j}}(I)\left( \alpha  \right) + {E_{2j}}^*(\alpha )}}{{{D_2}(I)\left( \alpha  \right)}}.$$
This is the derised result.
\end{proof}

The following lemma gives a lower bound for $\left| {{D_j}\left( I \right)\left( \alpha  \right)} \right|$ (defined in Lemma \ref{bod1}) when $\varphi$ satisfies the condition (W1). 
\begin{lemma} \label{bode2}
Let $\varphi  \in {L^1}\left( {0,T} \right)$ satisfies the condition (W1) then there exists a constant  $R\left( \varphi  \right) > 0$ such that for all $\alpha  = \left( {{\alpha _1},{\alpha _2},{\alpha _3}} \right) \in {\C^3}$, $\alpha \cdot \alpha<0$ and  ${\left| \alpha  \right|_0} \geqslant R\left( \varphi  \right)$,
$$\left| {{D_j}\left( I \right)\left( \alpha  \right)} \right| \geqslant \frac{1}{4}{\left| \alpha  \right|_0}{C\left( \varphi  \right)\min\left\{ {\dfrac{1}{{\sqrt {\mu}  }},\dfrac{1}{{\sqrt {\lambda  + 2\mu } }}} \right\}}~~\text{for}~j = 1,2.$$
\end{lemma}
\begin{proof} Denote ${k_1} = \sqrt {\lambda  + 2\mu }$ and ${k_2} = \sqrt {\mu}$. By the triangle inequality, one finds that
\bqq\label{bdt}
\left| { \frac{D_{j}\left(I\right)\left(\alpha\right)}{|\alpha|_0^2}} \right| &=& \left| - \intop_{0}^{T}\varphi\left(t\right)\frac{\sinh\Bigl(k_{j}\,{\left|\alpha\right|_0}\bigl(T-t\bigr)\Bigr)}{\cosh\left(k_{j}\,{\left|\alpha\right|_0}T\right)}dt\right|  \hfill\\
&\ge &  \left|\,\,\intop_{0}^{\Lambda\left(\varphi\right)}\varphi\left(t\right)\frac{\sinh\Bigl(k_{j}\,{\left|\alpha\right|_0}\bigl(T-t\bigr)\Bigr)}{\cosh\left(k_{j}{\left|\alpha\right|_0}T\right)}dt\right|\hfill\\
&~&- \left|\,\,\intop_{\Lambda\left(\varphi\right)}^{T}\varphi\left(t\right)\frac{\sinh\Bigl(k_{j}\,{\left|\alpha\right|_0}\bigl(T-t\bigr)\Bigr)}{\cosh\left(k_{j}\,{\left|\alpha\right|_0}T\right)}dt\right|. 
\eqq
We have, with $|\alpha|_0$ large, 
\bqq
\left|\,\,\intop_{\Lambda\left(\varphi\right)}^{T}\varphi\left(t\right)\frac{\sinh\Bigl(k_{j}\,{\left|\alpha\right|_0}\bigl(T-t\bigr)\Bigr)}{\cosh\left(k_{j}\,{\left|\alpha\right|_0}T\right)}dt\right| &\le& \left\Vert \varphi\right\Vert _{L^{1}}\frac{\sinh\biggl(k_{j}\,{\left|\alpha\right|_0}\Bigl(T-\Lambda\bigl(\varphi\bigr)\Bigr)\biggr)}{\cosh\left(k_{j}\,{\left|\alpha\right|_0}T\right)}\hfill\\ &\le& \frac{C(\varphi)}{4 \max\{k_1,k_2\} |\alpha|_0}.
\eqq
On the other hand, the condition (W1) implies that  
\bqq
&~&\left|\intop_{0}^{\Lambda\left(\varphi\right)}\varphi\left(t\right)\frac{\sinh\Bigl(k_{j}\,{\left|\alpha\right|_0}\bigl(T-t\bigr)\Bigr)}{\cosh\left(k_{j}\,{\left|\alpha\right|_0}T\right)}dt\right|
\ge  C\left(\varphi\right)\intop_{0}^{\Lambda\left(\varphi\right)}\frac{\sinh\Bigl(k_{j}\,{\left|\alpha\right|_0}\bigl(T-t\bigr)\Bigr)}{\cosh\left(k_{j}\,{\left|\alpha\right|_0}T\right)}dt\hfill\\
&=&\frac{C(\varphi)}{k_{j}\,{\left|\alpha\right|_0}}\Biggl(1-\frac{\cosh\biggl(k_{j}\,{\left|\alpha\right|_0}\Bigl(T-\Lambda\bigl(\varphi\bigr)\Bigr)\biggr)}{\cosh\left(k_{j}\,{\left|\alpha\right|_0}T\right)}\Biggr) \ge \frac{C(\varphi)}{2k_j|\alpha|_0}
\eqq
with $|\alpha|_0$ large. The desired result follows immediately from the above inequalities.
\end{proof}

The proof of the uniqueness below follows the argument in \cite{TDN}. We shall need a useful result of entire functions (see e.g. \cite{Ya}, Section 6.1). 
\begin{lemma}[Beurling]\label{beurling} Let $\phi$ be a non-constant entire function satisfying the condition: there exists a constant $k > 0$ such that ${M_\phi }\left( r \right) \leqslant k{e^r}$, for all $ r > 0$, where  ${M_\phi }\left( r \right) =\{\left| {\phi \left( z \right)}\right|: \left| z \right| = r\} $. Then 
$$\mathop {\limsup }\limits_{r \to \infty } \dfrac{{\ln \left| {\phi \left( r \right)} \right|}}
{r} \geqslant  - 1.$$
\end{lemma}

\begin{proof}[Proof of Theorem \ref{duynhatnghiem}] Suppose that $\left( {{u^1},{f^1}} \right)$ and $\left( {{u^2},{f^2}} \right)$ are two solutions to the system (\ref{1}-\ref{4}) with the same the data $I=(\varphi, X,g,h)$. Then $(u, f):= (u^1- u^2,f^1 - f^2)$ is a solution to (\ref{1}-\ref{4}) with data $(\varphi, 0,0,0)$. We shall show that $(u,f)=0$. 

Assume that $f\ne 0$, namely $f_j\ne 0$ for some $j\in \{1,2,3\}$. For any $n,p\in \N\cup\{0\}$, let us consider the entire function
$$z\mapsto {\psi _{n,p}}\left( z \right) = \int\limits_\Omega  {{f_j}(x)\cos \left( {iz{x_1}} \right)\cos \left( {n\pi {x_2}} \right)\cos \left( {p\pi {x_3}} \right)dx}.$$
Because 
$$\frac{d\psi _{n,p}}{dz}(im\pi)=\int\limits_\Omega  {ix_1 f_j(x)\sin \left( {m\pi{x_1}} \right)\cos \left( {n\pi {x_2}} \right)\cos \left( {p\pi {x_3}} \right)dx}
$$
and $\left\{ {\sin \left( {m\pi {x_1}} \right)\cos \left( {n\pi {x_2}} \right)\cos \left( {p\pi {x_3}} \right)} \right\}_{m\in \N, n,p\in \N \cup\{0\}}$ is an orthogonal basis on ${L^2}\left( \Omega  \right)$, there exist some $(n_0,p_0)$ such that $\psi _{{n_0},{p_0}}$ is non-constant. 

On the other hand, recall from Lemma \ref{bod1} that
\bq
\psi_{n_0,p_0}(r)=\int\limits_\Omega  {f_j(x){G(\alpha_r,x)}dx = \frac{{E_{1j}^*\left( \alpha_r  \right)}}
{{{D_1}\left( I \right)\left( \alpha_r  \right)}} + } \frac{{E_{2j}^*\left( \alpha_r  \right)}}
{{{D_2}\left( I \right)\left( \alpha_r  \right)}}, \label{dn1}
\eq
where $\alpha _r=( ir,n_0\pi ,p_0\pi)$.  Fix $\eps>0$ such that $\min\{T\sqrt{\lambda+2\mu},T\sqrt{\mu}\}>2+\eps$. Then it is straightforward to see, from the explicit formulas in Lemma 1 and the lower bound in Lemma \ref{bode2}, that 
$$|E_{ej}^*(\alpha_r)|\le C_1e^{-(1+\eps)r},|D_e(I)(\alpha_r)|\ge C_2r,~~e=1,2,$$
with $r>0$ large, where $C_1>0$ and $C_2>0$ are independent of $r$. Therefore, it follows from \eqref{dn1} that  
$$ |\psi_{n_0,p_0}(r)|\le 2C_1C_2^{-1}e^{-(1+\eps)r}r^{-1}$$
for $r>0$ large. This yields 
$$
\mathop {\limsup}\limits_{r \to \infty } \frac{{\ln \left| {\psi _{{n_0},{p_0}}\left( r \right)} \right|}}
{r} \le  - (1+\eps)$$
which is a contradiction to Lemma  \ref{beurling}. Thus ${f} \equiv 0$.

Now following the proof of Lemma \ref{bod1} up to the equations \eqref{PT6} and \eqref{PT7} ($\alpha^2:=\alpha_1^2+\alpha_2^3+\alpha_3^3$ needs not be negative at that time) we obtain  
\bq
&~& \frac{{{d^2}}}
{{d{t^2}}}\int\limits_\Omega  {\left( {{\alpha _j}{\alpha _1}{u_1}{G_1^j} + {\alpha _j}{\alpha _2}{u_2}{G_2^j} + {\alpha _j}{\alpha _3}{u_3}{G_3^j}} \right)dx}\nn \hfill \\
 &~&~+ \left( {\lambda  + 2\mu } \right)\alpha^2 \int\limits_\Omega  {\left( {{\alpha _j}{\alpha _1}{u_1}{G_1^j} + {\alpha _j}{\alpha _2}{u_2}{G_2^j} + {\alpha _j}{\alpha _3}{u_3}{G_3^j}} \right)} dx = 0, \label{cuoi1} \hfill\\
 &~& \frac{{{d^2}}}
{{d{t^2}}}\int\limits_\Omega  {\Bigl( {{\alpha ^2}{u_j}{G} - \bigl( {{\alpha _j}{\alpha _1}{u_1}{G_1^j} + {\alpha _j}{\alpha _2}{u_2}{G_2^j} + {\alpha _j}{\alpha _3}{u_3}{G_3^j}} \bigr)} \Bigr)} dx \nn\hfill \\
&~&~ + \mu \alpha^2 \int\limits_\Omega  {\Bigl( {{ {\alpha ^2}}{u_j}{G_j^j} - \bigl( {{\alpha _j}{\alpha _1}{u_1}{G_1^j} + {\alpha _j}{\alpha _2}{u_2}{G_2^j} + {\alpha _j}{\alpha _3}{u_3}{G_3^j}} \bigr)} \Bigr)} dx = 0\label{cuoi2}
\eq
for all $\alpha \in \mathbb{C}^3$. Note that $y \equiv  0$ is the unique solution to the differential equation
\[\left\{ \begin{gathered}
  y''\left( t \right) +{\eta ^2}y\left( t \right) = 0, \hfill \\
  y\left( 0 \right) = 0, \hfill \\
  y'\left( 0 \right) = 0, \hfill \\ 
\end{gathered}  \right.\]
where $\eta\ge 0$ independent of $t$. Applying this to \eqref{cuoi1} and \eqref{cuoi2} with $\alpha\in \mathbb{R}^3$ we get
$$
\int\limits_\Omega  {\left( {{\alpha _j}{\alpha _1}{u_1}{G_1^j} + {\alpha _j}{\alpha _2}{u_2}{G_2^j} + {\alpha _j}{\alpha _3}{u_3}{G_3^j}} \right)}dx = 0$$
and 
$$ \int\limits_\Omega  {\Bigl( {{ {\alpha ^2}}{u_j}{G} - \bigl( {{\alpha _j}{\alpha _1}{u_1}{G_1^j} + {\alpha _j}{\alpha _2}{u_2}{G_2^j} + {\alpha _j}{\alpha _3}{u_3}{G_3^j}} \bigr)} \Bigr)} dx = 0.$$
Adding the latter equations we obtain  
\begin{equation}\label{cuoic}
\int\limits_\Omega  {u_j(x,.)G(\alpha,x)} dx = 0 ~~\text{for all}~\alpha\in \R^3, \alpha\ne 0. 
\end{equation}

To get the same equation to (\ref{cuoic}) with $\alpha=0$, we can simply use identity \eqref{eq:Ak} with $\alpha=0$ to get 
$$\dfrac{{{d^2}}}
{{d{t^2}}}\,\int\limits_\Omega  {u_j(x,t)} dx = 0.$$
Since 
$$\left[ {\int\limits_\Omega  {u_j\left( {x,t} \right)} dx} \right]_{ t=0}= 0=\left[ {\frac{d}
{{dt}}\int\limits_\Omega  {u_j\left( {x, t } \right)dx} } \right]_{ t=0}$$ 
we have  
\begin{equation} 
\int\limits_\Omega  {u_j\left( {x,.} \right)} dx = 0. \label{cuoicc}
\end{equation} 
Putting \eqref{cuoic} and \eqref{cuoicc} together, we arrive at 
$$
\int\limits_\Omega  {u_jG} = \int\limits_\Omega  {u_j(x,.)\cos(\alpha_1 x_1)\cos(\alpha_2 x_2)\cos(\alpha_3 x_3)} dx= 0 ~~\text{for all}~\alpha\in \R^3.
$$
This implies that $u\equiv 0$ because $\{\cos(m\pi x_1)\cos(m\pi x_2)\cos(n\pi x_3)\}_{m,n,p\ge 0}$ forms an orthogonal basis of ${L^2}\left( \Omega  \right)$. 
\end{proof}

\section{Regularization}

Lemma \ref{bod1} and Lemma \ref{bode2} allow us compute $\F(f)(\alpha)$ approximately with $|\alpha|$ large. To recover $\F(f)(\alpha)$ with $|\alpha|$ small, we shall need the following interpolation inequality. This is an adaption of Lemma 4 in \cite{TDN}. The only difference is that we are working on a 3-dim problem, and hence we require more interpolation points.

\begin{lemma}[Interpolation inequality]\label{noisuy1}
Let ${B_r} = \left\{ { \pm z_j|z_j=5r + j, j = 1,2, \ldots ,24r} \right\}$ for some integer $r \geqslant 50$. Let $\omega:\C\to \C $ be an entire function such that $\left| {\omega \left( z \right)} \right| \leqslant A\,{e^{\left| z \right|}}$ for some constant $A$. Then for any function $g:{B_r} \to \C$ one has 
$$\mathop {\sup }\limits_{\left| z \right| \leqslant \pi r} \left| {\omega \left( z \right) - L\left[ {{B_r},g} \right]\left( z \right)} \right| \leqslant A\,{e^{ - r}} + 48\,r\,{e^{30r}}\mathop {\sup }\limits_{z \in {B_r}} \left| {\omega \left( z \right) - g\left( z \right)} \right|.$$
\end{lemma}
\begin{proof}
Fix $z \in \C$ with $\left| z \right| \leqslant \pi r$. We have the following residue formula at 49 simple poles $\{z\}\cup B_r$,  
$$\int\limits_{\xi = \left\{ {x \in \C:\left| x \right| = 50r} \right\}} {\frac{{\omega \left( x \right)}}
{{x - z}}\prod\limits_{j = 1}^{24r} {\frac{{{z^2} - z_j^2}}
{{{x^2} - z_j^2}}\,\,} } dx = 2\pi i\Bigl( {\omega \left( z \right) - L\left[ {{B_r},\omega } \right]\left( z \right)} \Bigr).$$
This gives the simple bound  
\bq \label{eq:simple-bound}
\left| {\omega \left( z \right) - L\left[ {{B_r},\omega } \right]\left( z \right)} \right| \leqslant 50r\mathop {\sup }\limits_{x \in \xi } \left\{ {\dfrac{{\left| {\omega \left( x \right)} \right|}}
{{\left| {x - z} \right|}}\prod\limits_{j = 1}^{24r} {\dfrac{{\left| {{z^2} - z_j^2} \right|}}
{{\left| {{x^2} - z_j^2} \right|}}} } \right\}.
\eq
Now we bound the right hand side of (\ref{eq:simple-bound}). For $x \in \C, |x|=50r$, we have $\left| {\omega \left( x \right)} \right| \leqslant A\,{e^{50r}}$, $\left| {x - z} \right| \geqslant \left( {50 - \pi } \right)r$ and
\begin{equation}\label{nsy2}\prod\limits_{j = 1}^{24r} {\frac{{\left| {{z^2} - z_j^2} \right|}}
{{\left| {{x^2} - z_j^2} \right|}}}  \leqslant \prod\limits_{j = 1}^{24r} {\frac{{{{\left| z \right|}^2} + z_j^2}}
{{{{\left| x \right|}^2} - z_j^2}} \leqslant } \prod\limits_{j = 1}^{24r} {\frac{{{{\left( {\pi r} \right)}^2} + z_j^2}}
{{{{\left( {50r} \right)}^2} - z_j^2}}.} \end{equation}
We shall show that 
\begin{equation}\label{nsy1}
\prod\limits_{j = 1}^{24r} {\frac{{{{\left( {\pi r} \right)}^2} + z_j^2}}
{{{{\left( {50r} \right)}^2} - z_j^2}}}  \leqslant \frac{{50 - \pi }}
{{50}}{e^{ - 51r}},\,\,\,\,\,\forall r \geqslant 50.
\end{equation}
Note that the function $x\mapsto v\left( x \right) = \ln \left( {\dfrac{{{{\left( {\pi r} \right)}^2} + x}}
{{{{\left( {50r} \right)}^2} - x}}} \right)$. $v\left( x \right)$ is increasing and concave in $\left[ {0,{{\left( {29r} \right)}^2}} \right]$. Using Jensen's inequality we get 
\begin{align*}
\sum\limits_{j = 1}^{24r} {v\left( {z_j^2} \right)}  &= \sum\limits_{h = 1}^6 {\left( {\sum\limits_{j = 4\left( {h - 1} \right)r + 1}^{4hr} {v\left( {z_j^2} \right)} } \right)}  \le 4r\sum\limits_{h = 1}^6 {v\left( {\frac{1}{{4r}}\sum\limits_{j = 4\left( {h - 1} \right)r + 1}^{4hr} {z_j^2} } \right)}\\ 
&= 4r\sum\limits_{h = 1}^6 {v\left( {\left( {24h + 16{h^2} + \frac{{31}}{3}} \right){r^2} + \left( {3 + 4h} \right)r + \frac{1}{6}} \right)}\\
&\le 4r\sum\limits_{h = 1}^6 {v\left( {\left( {24h + 16{h^2} + \frac{{31}}{3}} \right){r^2} + \left( {3 + 4h} \right)\frac{{{r^2}}}{{50}} + \frac{{{r^2}}}{{6 \times {{50}^2}}}} \right)} \\
&= 4r\sum\limits_{h = 1}^6 {\ln \left( {\frac{{{\pi ^2} + 24h + 16{h^2} + \frac{{31}}{3} + \frac{{3 + 4h}}{{50}} + \frac{1}{{6 \times {{50}^2}}}}}{{{{50}^2} - \left( {24h + 16{h^2} + \frac{{31}}{3} + \frac{{3 + 4h}}{{50}} + \frac{1}{{6 \times {{50}^2}}}} \right)}}} \right)} \\
&<  - 51r + \ln \left( {\frac{{50 - \pi }}{{50}}} \right).
\end{align*} 
Thus \eqref{nsy1} holds. Replacing \eqref{nsy2} and \eqref{nsy1} into (\ref{eq:simple-bound}) we obtain
\begin{equation}\label{nsy3}
\left| {\omega \left( z \right) - L\left[ {{B_r},\omega } \right]\left( z \right)} \right| \leqslant A\,{e^{ - r}},\,\,\,\forall z \in C,\,\,\left| z \right| \leqslant \pi r.
\end{equation}

Next, we observe that  
\begin{equation*}
\begin{split}
L\left[ {{B_r},\omega } \right]\left( z \right) - L\left[ {{B_r},g} \right]\left( z \right) = &\sum\limits_{j = 1}^{24r} {\left( {\prod\limits_{k \ne j} {\frac{{{z^2} - z_k^2}}
{{z_j^2 - z_k^2}}} } \right)} \frac{{z + {z_j}}}
{{2{z_j}}}\Bigl( {\omega \left( {{z_j}} \right) - g\left( {{z_j}} \right)} \Bigr)    \\
\phantom{L\left[ {{B_r},\omega } \right]\left( z \right) - L\left[ {{B_r},g} \right]\left( z \right) = }&+ \sum\limits_{j = 1}^{24r} {\left( {\prod\limits_{k \ne j} {\frac{{{z^2} - z_k^2}}
{{z_j^2 - z_k^2}}} } \right)} \frac{{z - {z_j}}}
{{ - 2{z_j}}}\Bigl( {\omega \left( { - {z_j}} \right) - g\left( { - {z_j}} \right)} \Bigr).
\end{split}
\end{equation*}
Therefore
\begin{equation}\label{nsy4}
\Bigl| {L\left[ {{B_r},\omega } \right]\left( z \right) - L\left[ {{B_r},g} \right]\left( z \right)} \Bigr| \leqslant 2\sigma \sum\limits_{j = 1}^{24r} {\left| {\prod\limits_{k \ne j} {\frac{{{z^2} - z_k^2}}
{{z_j^2 - z_k^2}}} } \right|,} 
\end{equation}
where $\sigma  = \mathop {\sup }\limits_{z \in {B_r}} \left| {\omega \left( z \right) - g\left( z \right)} \right|$. 

Now we bound the the right hand side of (\ref{nsy4}). We have  
\begin{equation*}\label{nsy5}
\begin{split}
\left| {\prod\limits_{k \ne j} {\frac{{{z^2} - z_k^2}}
{{z_j^2 - z_k^2}}} } \right| &\leqslant \prod\limits_{k \ne j} {\frac{{{{\left| z \right|}^2} + z_k^2}}
{{\left| {z_j^2 - z_k^2} \right|}}}  = \prod\limits_{k \ne j} {\frac{{{{\left| z \right|}^2} + z_k^2}}
{{\left( {{z_j} + {z_k}} \right){z_k}}} \cdot \frac{{{z_k}}}
{{\left| {{z_j} - {z_k}} \right|}}}\\
&\leqslant \prod\limits_{k \ne j} {\frac{{{z_k}}}
{{\left| {{z_j} - {z_k}} \right|}}}    \leqslant \frac{{\left( {5r + 2} \right)\left( {5r + 3} \right) \ldots \left( {29r} \right)}}
{{\left( {j - 1} \right)!\left( {24r - j} \right)!}}\\
&\leqslant \frac{{\left( {5r + 2} \right)\left( {5r + 3} \right) \ldots \left( {29r} \right)}}
{{\left( {12r - 1} \right)!\left( {12r} \right)!}}=:J(r). 
\end{split} 
\end{equation*}
Since $J\left( 1 \right) \leqslant {e^{30}}$ and, by direct expansion,  
\begin{equation*}
\begin{split}
 \frac{{J\left( {r + 1} \right)}}
{{J\left( r \right)}} &= \frac{{\left( {29r + 1} \right)\left( {29r + 2} \right) \ldots \left( {29r + 29} \right)}}
{{\left( {5r + 2} \right) \ldots \left( {5r + 6} \right) \cdot {{\left[ {\left( {12r + 1} \right) \ldots \left( {12r + 11} \right)} \right]}^2} \cdot \left( {12r} \right) \cdot \left( {12r + 12} \right)}}  \\
 &\leqslant \frac{{{{29}^{29}}}}
{{{5^5} \cdot {{12}^{24}}}} < {e^{30}}, 
\end{split}
\end{equation*} 
we conclude that $J\left( r \right) \leqslant {e^{30r}}$ for all $r\ge 1$. Thus (\ref{nsy4}) reduces to 
\bq \label{eq:second-triangle}
\Bigl| {L\left[ {{B_r},\omega } \right]\left( z \right) - L\left[ {{B_r},g} \right]\left( z \right)} \Bigr| \leqslant 48r\sigma e^{30r}.
\eq

The desired result follows from (\ref{nsy3}), (\ref{eq:second-triangle}) and  the triangle inequality.
\end{proof}

The last of our preparation to prove the regularization result is the following useful lemma on the truncated Fourier series. The proof is elementary and we refer to Lemma 5 in \cite{TMDN}  for details (in \cite{TMDN} the authors dealt with the 2-dimensional case, but the proof for 3-dimensional case is essentially the same).  

\begin{lemma} \label{le:lemma4}For each $w\in L^2(\Omega)$ and $r>0$, we define
\[
\Gamma _r (w)(x) = \sum\limits_{0\le m,n,p\le r} {\kappa (m,n,p)\F(w) (m\pi ,n\pi,p\pi)\cos (m\pi x_1)\cos (n\pi x_2)\cos(p\pi x_3)} 
\]
Then
\begin{itemize}
\item[(i)] $\Gamma_r (w) \to w$ in $L^2 (\Omega )$ as $r\to \infty$.

\item[(ii)] If $w\in H^1(\Omega)$ then $\Gamma _r (w)\to w$ in $H^1(\Omega)$ and 
\[
\left\| {\Gamma_r (w) - w} \right\|_{L^2 (\Omega )}  \leqslant \frac{1}
{{\pi \sqrt{r}}}\left\| w \right\|_{H^1 (\Omega )} .
\]

\item[(iii)] If $w\in H^2(\Omega)$ then  
\[
\left\| {\Gamma_r (w) - w} \right\|_{H^1 (\Omega )}  \leqslant \frac{4}
{{\sqrt[4] r}}\left\| w \right\|_{H^2 (\Omega )} .
\]
\end{itemize}
\end{lemma}

We are now ready to prove the main result. 
\begin{proof}[Proof of Theorem \ref{ch2chinhhoanghiem}] We shall first estimate the error $f_j^\varepsilon  - \Gamma _{{r_\varepsilon }}f_j^0$, where $\Gamma_r$ is defined in Lemma \ref{le:lemma4}. Then we compare $\Gamma _{{r_\varepsilon }}f_j^0$ and $f^j_0$ by employing Lemma \ref{le:lemma4}. The conclusion follows from the triangle inequality. In the following, $C, C_1,C_2$ always stand for constants depending only on the exact data, and saying that $\eps>0$ small enough means that $\eps\in (0,\eps_0)$ for some constant $\eps_0>0$  depending only on the exact data.
\\\\
{\bf Step 1:} Bound on $\left| {F_j^\varepsilon \left( {m,n,p} \right) - {\F}\left( {f_j^0} \right)\left( {m\pi ,n\pi ,p\pi } \right)} \right|$ for $m,n,p \in \left[ {0,{r_\varepsilon }} \right]$.

Applying Lemma \ref{noisuy1} to $g\left( z \right) = H_j^\varepsilon \left( {z,n\pi ,p\pi } \right)$ and  
$$  \omega \left( z \right) = \int\limits_\Omega  {f_j^0\left( {{x_1},{x_2},{x_3}} \right)\cos \left( { - iz{x_1}} \right)\cos \left( {n\pi {x_2}} \right)\cos \left( {p\pi {x_3}} \right)d{x_1}d{x_2}d{x_3}}  $$
we find that for $m,n,p \in \left[ {0,{r_\varepsilon }} \right]$ and $\varepsilon>0$ small enough
\bq\label{noisuy2}
 &~& \left| {F_j^\varepsilon \left( {m,n,p} \right) - {\F}\left( {f_j^0} \right)\left( {m\pi ,n\pi ,p\pi } \right)} \right| \nn\hfill\\
&=& \Bigl| {L\left[ {{B_{{r_\varepsilon }}},g} \right]\left( {im\pi } \right) - \omega \left( {im\pi } \right)} \Bigr|  \nn\hfill \\
&\leqslant & {\left\| {f_j^0} \right\|_{{L^1}\left( \Omega  \right)}}{e^{ - {r_\varepsilon }}} + 48\,{r_\varepsilon }\,{e^{30{r_\varepsilon }}}\mathop {\sup }\limits_{z \in {B_{{r_\varepsilon }}}} \left| {\omega \left( z \right) - g\left( z \right)} \right|.
\eq

The error $\left| {\omega \left( z \right) - g\left( z \right)} \right|$ can be bound by direct computation. Indeed, for  $z \in {B_{{r_\varepsilon }}}$, we have  
\bq\label{noisuy3}
  \left| {\omega \left( z \right) - g\left( z \right)} \right| &=& \left| {{\F}\left( {f_j^0} \right)\left( { - iz,n\pi ,p\pi } \right) - H_j^\varepsilon \left( {z,n\pi ,p\pi } \right)} \right| \nn \hfill\\
&\leqslant & \left| {{\F}\left( {f_j^0} \right)\left( { - iz,n\pi ,p\pi } \right) - H_j^0\left( {z,n\pi ,p\pi } \right)} \right| \nn \hfill\\
&~&~~~+ \left| {H_j^0\left( {z,n\pi ,p\pi } \right) - H_j^\varepsilon \left( {z,n\pi ,p\pi } \right)} \right| \nn \hfill\\
&\leqslant & \left| {\frac{{E_{1j}^*\left( { - iz,n\pi ,p\pi } \right)}}
{{{D_1}\left( {{I_0}} \right)\left( { - iz,n\pi ,p\pi } \right)}} + \frac{{E_{2j}^*\left( { - iz,n\pi ,p\pi } \right)}}
{{{D_2}\left( {{I_0}} \right)\left( { - iz,n\pi ,p\pi } \right)}}} \right| \nn\hfill \\
&~&~~~+ \left| {\frac{{{E_{1j}}\left( {{I_0}} \right)\left( { - iz,n\pi ,p\pi } \right)}}
{{{D_1}\left( {{I_0}} \right)\left( { - iz,n\pi ,p\pi } \right)}} - \frac{{{E_{1j}}\left( {{I_\varepsilon }} \right)\left( { - iz,n\pi ,p\pi } \right)}}
{{{D_1}\left( {{I_\varepsilon }} \right ) \left( { - iz,n\pi ,p\pi } \right)}}} \right| \nn \hfill\\
&~&~~~+\left| {\frac{{{E_{2j}}\left( {{I_0}} \right)\left( { - iz,n\pi ,p\pi } \right)}}
{{{D_2}\left( {{I_0}} \right)\left( { - iz,n\pi ,p\pi } \right)}} - \frac{{{E_{2j}}\left( {{I_\varepsilon }} \right)\left( { - iz,n\pi ,p\pi } \right)}}
{{{D_2}\left( {{I_\varepsilon }} \right)\left( { - iz,n\pi ,p\pi } \right)}}} \right|. 
\eq
It is straightforward to see from the explicit formulas of  ${E_{1j}}$, ${E_{2j}}$, $E_{1j}^*$,  $E_{2j}^*$ that
\bqq
 \left| {{E_{1j}}\left( {{I_\varepsilon }} \right)\left( { - iz,n\pi ,p\pi } \right) - {E_{1j}}\left( {{I_0}} \right)\left( { - iz,n\pi ,p\pi } \right)} \right| 
&\le& {C_1}r_\varepsilon ^3\,{e^{29{r_\varepsilon }}}\,\varepsilon, \hfill\\
  \left| {{E_{2j}}\left( {{I_\varepsilon }} \right)\left( { - iz,n\pi ,p\pi } \right) - {E_{2j}}\left( {{I_0}} \right)\left( { - iz,n\pi ,p\pi } \right)} \right| 
&\le& {C_1}r_\varepsilon ^3\,{e^{29{r_\varepsilon }}}\,\varepsilon, \hfill\\
 \left| {{E_{1j}}\left( {{I_0}} \right)\left( { - iz,n\pi ,p\pi } \right)} \right| 
&\le& {C_1}\,r_\varepsilon ^3\,{e^{29{r_\varepsilon }}},\hfill\\
\left| {{E_{2j}}\left( {{I_0}} \right)\left( { - iz,n\pi ,p\pi } \right)} \right| 
&\le& {C_1}\,r_\varepsilon ^3\,{e^{29{r_\varepsilon }}},  \\
  \left| {E_{1j}^*\left( { - iz,n\pi ,p\pi } \right)} \right| 
&\le& \frac{{{C_1}\,r_\varepsilon ^3\,{e^{29{r_\varepsilon }}}}}{{{e^{\sqrt 5 \sqrt {  \lambda  + 2\mu } \,\,T\,{r_\varepsilon }}}}},\hfill\\
  \left| {E_{2j}^*\left( { - iz,n\pi ,p\pi } \right)} \right| 
&\le& \frac{{{C_1}\,r_\varepsilon ^3\,{e^{29{r_\varepsilon }}}}}
{{{e^{\sqrt 5 \sqrt {  \mu } \,\,T\,{r_\varepsilon }}}}}.
\eqq
On the other hand, Lemma \ref{bode2} ensures that $
 C_2 r_\varepsilon \le \left| {{D_j}\left( {{I_0}} \right)\left( { - iz,n\pi ,p\pi } \right)} \right| \le {C_1}r_\varepsilon^2 $ and
 \bqq
 \left| {{D_j}\left( {{I_\varepsilon }} \right)\left( { - iz,n\pi ,p\pi } \right)} \right| &\ge& \left| {{D_j}\left( {{I_0}} \right)\left( { - iz,n\pi ,p\pi } \right)} \right|  \hfill\\
&~&- \,\,\left| {{D_j}\left( {{I_\varepsilon }} \right)\left( { - iz,n\pi ,p\pi } \right) - {D_j}\left( {{I_0}} \right)\left( { - iz,n\pi ,p\pi } \right)} \right| \hfill\\
 &\ge& {C_2}{r_\varepsilon } - {C_1}r_\varepsilon ^2\,\varepsilon, 
 \eqq
for all $z \in {B_{{r_\varepsilon }}}$ and $\eps>0$ small enough. Therefore, from \eqref{noisuy2} and \eqref{noisuy3} we get
\begin{equation*}\label{btchoa1}
\begin{gathered}
  \left| {F_j^\varepsilon \left( {m,n,p} \right) - {\F}\left( {f_j^0} \right)\left( {m\pi ,n\pi ,p\pi } \right)} \right| \leqslant {\left\| {f_j^0} \right\|_{{L^1}\left( \Omega  \right)}}\,{\varepsilon ^{\frac{1}
{{60}}}} \,\,+ \,48\, \times \hfill \\[0.2cm]\noindent
~~~~\times \left[ {\frac{{{C_1}{{\Bigl( {\frac{1}
{{30}}\ln \left( {{\varepsilon ^{ - 1}}} \right)} \Bigr)}^3}}}
{{{C_2}{\varepsilon ^{ - \frac{1}
{{60}}\left( {\sqrt 5 \sqrt {  \lambda  + 2\mu } \,\,T - 59} \right)}}}} + \frac{{{C_1}{{\Bigl( {\frac{1}
{{30}}\ln \left( {{\varepsilon ^{ - 1}}} \right)} \Bigr)}^3}}}
{{{C_2}{\varepsilon ^{ - \frac{1}
{{60}}\left( {\sqrt 5 \sqrt {  \mu } \,\,T - 59} \right)}}}} + \frac{{4C_1^2{{\Bigl( {\frac{1}
{{30}}\ln \left( {{\varepsilon ^{ - 1}}} \right)} \Bigr)}^4}{e^{59}}{\varepsilon ^{\frac{1}
{{60}}}}}}
{{{C_2}\Bigl( {{C_2} - \frac{1}
{{30}}{C_1}\ln \left( {{\varepsilon ^{ - 1}}} \right)\varepsilon } \Bigr)}}} \right]. \hfill \\ 
\end{gathered}
\end{equation*}
Moreover, the condition (W2') gives 
$$
  \dfrac{1}
{{60}} < \dfrac{1}
{{60}}\left( {\sqrt 5 \sqrt {  \lambda  + 2\mu } \,\,T - 59} \right), \;\; \dfrac{1}
{{60}} < \dfrac{1}
{{60}}\left( {\sqrt 5 \sqrt {  \mu } \,\,T - 59} \right).$$
Thus we obtain
\begin{equation}\label{btchoa2}
\left| {F_j^\varepsilon \left( {m,n,p} \right) - {\F}\left( {f_j^0} \right)\left( {m\pi ,n\pi ,p\pi } \right)} \right| \leqslant {C}{\Bigl( {\ln \left( {{\varepsilon ^{ - 1}}} \right)} \Bigr)^4}{\varepsilon ^{\frac{1}
{{60}}}},
\end{equation}
for all $\varepsilon>0$ small enough. 
\\\\
{\bf Step 2:} Conclusion. 

(i) Using the Parseval equality and the estimate \eqref{btchoa2} one gets 
\bq \label{eq:fe-Gf}
\left\| {f_j^\varepsilon  - \Gamma _{{r_\varepsilon }} f_j^0} \right\|_{L^2\left( \Omega  \right)}^2 
&=& \sum\limits_{0\, \le \,m,\,n,\,p\,\, \le \,{r_\varepsilon }} {\kappa\left( {m,n,p} \right){{\left| {F_j^\varepsilon \left( {m,n,p} \right) - {\F}\left( {f_j^0} \right)\left( {m\pi ,n\pi ,p\pi } \right)} \right|}^2}.}  \nn \hfill\\
&\le& C(\ln (\varepsilon ^{ - 1}))^{11} \varepsilon ^{\frac{1}{30}}.
\eq
Since $\left\| {f_j^\varepsilon  - \Gamma _{{r_\varepsilon }}f_j^0} \right\|_{L^2(\Omega)}\to 0$ by (\ref{eq:fe-Gf}) and $\left\| {\Gamma _{{r_\varepsilon }}f_j^0-f_j^0} \right\|_{L^2(\Omega)}\to 0$ by Lemma \ref{le:lemma4}, we obtain $ ||f_j^\varepsilon-f_j^0||_{L^2(\Omega)}\to 0$ as $\eps\to 0$ by the triangle inequality.

(ii) Assume that $f_j^0\in H^1(\Omega)$. Then by Lemma \ref{le:lemma4} one has $\left\| {\Gamma _{{r_\varepsilon }}f_j^0-f_j^0} \right\|_{H^1(\Omega)}\to 0$ and $\left\| {\Gamma _{{r_\varepsilon }}f_j^0-f_j^0} \right\|_{L^2(\Omega)}\le C(\ln (\eps^{-1}))^{-\frac{1}{2}}.$ Combining the latter estimate with (\ref{eq:fe-Gf}), we conclude using the triangle inequality that 
$${\left\| {f_j^\varepsilon  - f_j^0} \right\|_{{L^2}\left( \Omega  \right)}} \leqslant {C}{\Bigl( {\ln \left( {{\varepsilon ^{ - 1}}} \right)} \Bigr)^{ - \frac{1}{2}}}$$
for $\eps>0$ small enough.

Moreover, using (\ref{btchoa2}) and the Parseval equality for $H^1$  we have
\bq \label{eq:fe-Gf-H1}
&~&\left\| {f_j^\varepsilon  - \Gamma _{{r_\varepsilon }} f_j^0} \right\|_{H^1\left( \Omega  \right)}^2\nn \hfill\\
&=& \sum\limits_{0\, \le \,m,\,n,\,p\,\, \le \,{r_\varepsilon }} {\kappa\left( {m,n,p} \right) (1+\pi^2(m^2+n^2+p^2)){{\left| {F_j^\varepsilon \left( {m,n,p} \right) - {\F}\left( {f_j^0} \right)\left( {m\pi ,n\pi ,p\pi } \right)} \right|}^2}.}  \nn \hfill\\
&\le& C(\ln (\varepsilon ^{ - 1}))^{13} \varepsilon ^{\frac{1}{30}} \to 0.
\eq
Thus $\left\| {f_j^\varepsilon  - f_j^0} \right\|_{H^1\left( \Omega  \right)}\to 0$ by the triangle inequality (in $H^1$-norm). 

(iii) Assume that $f_j^0\in H^2(\Omega)$. Then by Lemma \ref{le:lemma4} one has $$\left\| {\Gamma _{{r_\varepsilon }}f_j^0-f_j^0} \right\|_{H^1(\Omega)}\le C(\ln (\eps^{-1}))^{-1/4}.$$ This estimate together with (\ref{eq:fe-Gf-H1}) and the triangle inequality yield
$$\left\| {f_j^\varepsilon  - f_j^0} \right\|_{H^1\left( \Omega  \right)}\le C(\ln (\eps^{-1}))^{-1/4}$$
for $\eps>0$ small enough.
\end{proof}

\section{Numerical example}

In this section, we test our regularization process in an explicit example. Choose $\mu  =   1$, $\lambda  = -1$, $T = 30$ and consider the system (\ref{1}-\ref{4}) with the exact data ${I_0}=\left( {{\varphi ^0},{X^0},{g^0},{h^0}} \right)$ given by
\begin{equation*}
\begin{split}
{\varphi _0}\left( t \right) =&23{\pi ^2}\cos \left( {\pi t} \right),~h_1^0 = h_2^0 = h_3^0 = 0,\\
X_1^0\left( {{x_1},{x_2},{x_3},t} \right) =& \cos \left( {\pi t} \right)\Bigl( - 4\pi \sin \left( {2\pi {x_2}} \right)\sin \left( {2\pi {x_3}} \right){n_1}\\
&- 2\pi \sin \left( {4\pi {x_1}} \right)\sin \left( {2\pi {x_3}} \right){n_2} - 2\pi \sin \left( {4\pi {x_1}} \right)\sin \left( {2\pi {x_2}} \right){n_3}\Bigr),\\
X_2^0\left( {{x_1},{x_2},{x_3},t} \right) =& \cos \left( {\pi t} \right)\Bigl( - 2\pi \sin \left( {4\pi {x_2}} \right)\sin \left( {2\pi {x_3}} \right){n_1}\\
&- 4\pi \sin \left( {2\pi {x_1}} \right)\sin \left( {2\pi {x_3}} \right){n_2} - 2\pi \sin \left( {2\pi {x_1}} \right)\sin \left( {4\pi {x_2}} \right){n_3}\Bigr),\\
X_3^0\left( {{x_1},{x_2},{x_3},t} \right) =& \cos \left( {\pi t} \right)\Bigl( - 2\pi \sin \left( {2\pi {x_2}} \right)\sin \left( {4\pi {x_3}} \right){n_1}\\
&- 2\pi \sin \left( {2\pi {x_1}} \right)\sin \left( {4\pi {x_3}} \right){n_2} - 4\pi \sin \left( {2\pi {x_1}} \right)\sin \left( {2\pi {x_2}} \right){n_3}\Bigr),\\
g_1^0\left( {{x_1},{x_2},{x_3}} \right) =& \sin \left( {4\pi {x_1}} \right)\sin \left( {2\pi {x_2}} \right)\sin \left( {2\pi {x_3}} \right),\\
g_2^0\left( {{x_1},{x_2},{x_3}} \right) =& \sin \left( {2\pi {x_1}} \right)\sin \left( {4\pi {x_2}} \right)\sin \left( {2\pi {x_3}} \right),\\
g_3^0\left( {{x_1},{x_2},{x_3}} \right) =& \sin \left( {2\pi {x_1}} \right)\sin \left( {2\pi {x_2}} \right)\sin \left( {4\pi {x_3}} \right),\\
\end{split}
\end{equation*}

Note that in this example the conditions (W1) and (W2') are satisfied. The exact solution of the system (\ref{1}-\ref{4}) can be computed explicitly  
\begin{equation*}
\begin{split}
{u_1^0}\left( {{x_1},{x_2},{x_3},t} \right) =& \cos \left( {\pi t} \right)\sin \left( {4\pi {x_1}} \right)\sin \left( {2\pi {x_2}} \right)\sin \left( {2\pi {x_3}} \right),\\
{u_2^0}\left( {{x_1},{x_2},{x_3},t} \right) =& \cos \left( {\pi t} \right)\sin \left( {2\pi {x_1}} \right)\sin \left( {4\pi {x_2}} \right)\sin \left( {2\pi {x_3}} \right),\\
{u_3^0}\left( {{x_1},{x_2},{x_3},t} \right) =& \cos \left( {\pi t} \right)\sin \left( {2\pi {x_1}} \right)\sin \left( {2\pi {x_2}} \right)\sin \left( {4\pi {x_3}} \right),\\
{f_1^0}\left( {{x_1},{x_2},{x_3}} \right) =& \sin \left( {4\pi {x_1}} \right)\sin \left( {2\pi {x_2}} \right)\sin \left( {2\pi {x_3}} \right),\\
{f_2^0}\left( {{x_1},{x_2},{x_3}} \right) =& \sin \left( {2\pi {x_1}} \right)\sin \left( {4\pi {x_2}} \right)\sin \left( {2\pi {x_3}} \right),\\
{f_3^0}\left( {{x_1},{x_2},{x_3}} \right) =& \sin \left( {2\pi {x_1}} \right)\sin \left( {2\pi {x_2}} \right)\sin \left( {4\pi {x_3}} \right).
\end{split}
\end{equation*}

Now, we consider the disturbed data, for $n\in \N$, $j \in \left\{ {1,2,3} \right\}$, 
\begin{equation*}
\begin{split}
{\varphi _n} &= \varphi ,~h_j^n = 0,\\
X_j^n &= X_j^0 + \dfrac{{ - 2\pi \cos \left( {\pi t} \right)}}{{\sqrt n }}\Bigl(\sin \left( {2n\pi {x_2}} \right)\sin \left( {2n\pi {x_3}} \right){\text{\bf n}_1} + \\
&~~~+ \sin \left( {2n\pi {x_1}} \right)\sin \left( {2n\pi {x_3}} \right){\text{\bf n}_2} + \sin \left( {2n\pi {x_1}} \right)\sin \left( {2n\pi {x_2}} \right){\text{\bf n}_3}\Bigr),\\
g_j^n &= g_j^0 + \dfrac{{\sin \left( {2n\pi {x_1}} \right)\sin \left( {2n\pi {x_2}} \right)\sin \left( {2n\pi {x_3}} \right)}}{{{n^{\frac{3}{2}}}}}.
\end{split}
\end{equation*}
The disturbed solution of the system (\ref{1}-\ref{4}) with the disturbed data  is 
\bqq
u_j^n &=& u_j^0 + \dfrac{{\cos \left( {\pi t} \right)\sin \left( {2n\pi {x_1}} \right)\sin \left( {2n\pi {x_2}} \right)\sin \left( {2n\pi {x_3}} \right)}}{{\sqrt {{n^3}} }},\hfill\\
f_j^n &=& f_j^0 + \dfrac{{\left( { - 1 + 12{n^2}} \right)\sin \left( {2n\pi {x_1}} \right)\sin \left( {2n\pi {x_2}} \right)\sin \left( {2n\pi {x_3}} \right)}}{{23\sqrt {{n^3}} }}.
\eqq

We can see that the error of the data is small 
$$
{\left\| {g_j^n - g_j^0} \right\|_{{L^2}\left( \Omega  \right)}} = \frac{{\sqrt 2 }}{{4\sqrt {{n^3}} }},~{\left\| {X_j^n - X_j^0} \right\|_{{L^1}\left( {0,T,{L^1}\left( {\partial \Omega } \right)} \right)}} \le \frac{{360\pi }}{{\sqrt n }}
$$
(in fact, they even converge in the uniform norm). However, the error of the solution is large since 
$${\left\| {f_j^n - f_j^0} \right\|_{{L^2}\left( \Omega  \right)}} = \frac{{\sqrt 2 }}{4}\sqrt {\frac{{144{n^4} - 24{n^2} + 1}}{{529{n^3}}}}  \to \infty $$
as $n \to \infty $. Thus the problem is ill-posed and a regularization is necessary.

Now, we apply our regularization procedure for $\varepsilon  =0.01$ and the disturbed data with $n=10$. The resulting regularized solution is 
\begin{equation*}
\begin{split}
f_1^\varepsilon \left( {{x_1},{x_2},{x_3}} \right) \approx&  \;0.035 - 0.063\cos \left( {\pi {x_1}} \right) - 0.156\cos \left( {\pi {x_2}} \right) - 0.455\cos \left( {\pi {x_3}} \right)\\
 &+\, 0.033\cos \left( {\pi {x_1}} \right)\cos \left( {\pi {x_2}} \right) + 0.027\cos \left( {\pi {x_1}} \right)\cos \left( {\pi {x_3}} \right) \\
 &+ 0.15\cos \left( {\pi {x_2}} \right)\cos \left( {\pi {x_3}} \right) -  0.005\cos \left( {\pi {x_1}} \right)\cos \left( {\pi {x_2}} \right)\cos \left( {\pi {x_3}} \right).
\end{split}
\end{equation*}
The error between the regularized solution and the exact solution is 
$$\left\| {f_1^\varepsilon  - f_1^0} \right\|_{{L^2}\left( \Omega  \right)}^2 \approx 0.273.$$
To see the effect of our regularization, note that the corresponding disturbed solution (with data error $\eps=0.01$) causes an extremely large error $\left\| {f_j^{10} - f_j^0} \right\|_{{L^2}\left( \Omega  \right)}^2 \approx 3.4 \times {10^5}$.
\\\\
{\bf Acknowledgments.} The work was done when P.T.Thuc was a Master student in Ho Chi Minh City University of Education.


Dang Duc Trong, Faculty of Mathematics, Vietnam National University, Ho Chi Minh City, Vietnam. {E-mail:} ddtrong@math.hcmus.edu.vn
\vspace{10pt}

Phan Thanh Nam, Department of Mathematical Sciences, University of Copenhagen, Denmark. {E-mail:} ptnam@math.ku.dk 

\vspace{10pt}

Phung Trong Thuc, Department of Mathematics, Van Lang University, Ho Chi Minh City, Vietnam. {E-mail:} phungtrongthuc@vanlanguni.edu.vn

\end{document}